\documentclass[a4]{article}
\pdfoutput=1
\usepackage{pallatinored}
\usepackage{verbatim}

\renewcommand{\refname}{References}

\newcommand{\sprod}[2]{\left\langle #1, #2\right\rangle}

\newcommand{\interior}[1]{\text{int }#1}

\renewcommand{\Im}{\text{Im }}
\newcommand{\dom}{\text{dom }}

\newcommand{\conv}{\text{conv }}

\renewcommand{\SS}{\mathcal{S}}

\begin{document}

\begin{center}
{\huge\textbf{\itshape Uneven Splitting of\\ Ham Sandwiches}}
\medskip

{\large {\itshape Felix Breuer}}

\end{center}
\bigskip

\begin{abstract} Let $\mu_1,\ldots,\mu_n$ be continuous probability measures on $\RR^n$ and $\alpha_1,\ldots,\linebreak \alpha_n\in[0,1]$. When does there exist an oriented hyperplane $H$ such that the positive half-space $H^+$ has $\mu_i(H^+)=\alpha_i$ for all $i\in[n]$? It is well known that such a hyperplane does not exist in general. The famous ham sandwich theorem  states that if $\alpha_i=\frac{1}{2}$ for all $i$, then such a hyperplane always exists.

In this paper we give sufficient criteria for the existence of $H$ for general $\alpha_i\in[0,1]$. Let $f_1,\ldots,f_n:S^{n-1}\rar\RR^n$ denote auxiliary functions with the property that for all $i$ the unique hyperplane $H_i$ with normal $v$ that contains the point $f_i(v)$ has $\mu_i(H_i^+)=\alpha_i$. Our main result is that if $\Im f_1,\ldots,\Im f_n$ are bounded and can be separated by hyperplanes, then there exists a hyperplane $H$ with $\mu_i(H^+)=\alpha_i$ for all $i$. This gives rise to several corollaries, for instance if the supports of $\mu_1,\ldots,\mu_n$ are bounded and can be separated by hyperplanes, then $H$ exists for any choice of $\alpha_1,\ldots,\alpha_n\in [0,1]$. We also obtain results that can be applied if the supports of $\mu_1,\ldots,\mu_n$ overlap.
\end{abstract}

\section{Introduction}
The ham sandwich theorem states that any ham sandwich in $\RR^3$ consisting of ham, cheese and bread can be split by a hyperplane such that ham, cheese and bread are simultaneously split in half. In general it reads as follows.

\begin{theorem}
Let $\mu_1,\ldots,\mu_n$ be continuous probability measures on $\RR^n$. Then there exists an oriented hyperplane $H$ such that 
$$\mu_i(H^+)=\frac{1}{2}\text{ for all $i\in\{1,\ldots,n\}=:[n]$,}$$
where $H^+$ denotes the positive half-space corresponding to $H$.
\end{theorem}

The ham sandwich theorem was posed as a problem by Steinhaus in the Scottish Book \cite[Problem 123]{sb}. A proof for $n=3$ was given by Banach \cite{s-38} using the Borsuk-Ulam theorem. This proof was generalized to arbitrary $n$ by Stone and Tukey in \cite{st-42}. For an account of the early history of the ham sandwich theorem we refer to \cite{bz-04}. For a reference on the Borsuk-Ulam theorem, including its application to the ham sandwich theorem and an overview of related results, we recommend \cite{m-ubut}.

The question we want to address is this: Let $\mu_1,\ldots,\mu_n$ be continuous probability measures on $\RR^n$ and $\alpha_1,\ldots,\alpha_n\in[0,1]$. When does there exist a hyperplane $H$ such that $\mu_i(H^+)=\alpha_i$  for all $i\in[n]$?

It is well known that such a hyperplane does not exist in general \cite{st-42}, \cite{m-ubut}. We give examples for this in section \ref{sec:examples}. Interestingly the equivariant methods used to obtain equipartition results cannot be easily applied to obtain partitions according to other ratios. 

There is much literature on the ham sandwich theorem and related partitioning results, see \cite{m-ubut} for an overview and e.g.\ \cite{bm-02},\cite{bks-00},\cite{dol-92},\cite{grunb-60},\cite{gh-05},\cite{rado-46},\cite{ramos-96},\cite{zv-2001},\cite{zv-1990}. However, most results assert the existence of equipartitions. Some sources in which the existence of partitions into parts of different measure are shown, are the following.

A $k$-fan in $\RR^2$ is a set of $k$ rays emanating from a point $p$. The regions in between two adjacent rays are called sectors $\sigma_1,\ldots,\sigma_k$. In the case of continuous measures, a $k$-fan $(\alpha_1,\ldots,\alpha_k)$-partitions measures $\mu_1,\ldots,\mu_n$ if $\mu_i(\sigma_j)=\alpha_j$ for all $j$ and $i$. Bárány and Matoušek \cite{bm-02} show that any two measures on $\RR^2$ can be $\alpha$-partitioned by a $2$-fan for all $\alpha$ and by a $4$-fan for $\alpha=(\frac{2}{5},\frac{1}{5},\frac{1}{5},\frac{1}{5})$.

Vrećica and Živaljević \cite{zv-2001} show that for any continuous measure, any $\alpha$ with $\sum_i\alpha_i =1$, any non-degenerate simplex $\Delta$ in $\RR^n$ and any point $a\in\text{int }\Delta$ there is a vector $v\in\RR^n$ such that the cones $R_i$ with apex $a$ generated by the facets of $\Delta$ satisfy $\mu(R_i+v)=\alpha_i$ for all $i$.

Rado \cite{rado-46} shows that for any measure $\mu$ on $\RR^n$ there exists a point $x$ such that $\mu(H^+)\geq\frac{1}{n+1}$ for any half-space $H^+$ with $x\in H^+$. Živaljević and Vrećica \cite{zv-1990} give a unification of Rado's theorem with the ham sandwich theorem: For measures $\mu_1,\ldots,\mu_k$ on $\RR^n$ there is a $(k-1)$-dimensional affine subspace $A$ such that $\mu(H^+)\geq\frac{1}{n-k+2}$ for any half-space $H^+$ with $A\subset H^+$. See also \cite{dol-92}.

Stojmenović \cite{stoj1} gives an algorithm for finding a hyperplane bisecting the volume of three convex polygons in $\RR^3$, that can be separated by hyperplanes (see section \ref{sec:separation}). Stojmenović remarks that this algorithm works in any dimension, for other measures (not only volume) and other proportions of splitting. In the journal version \cite{stoj2} of that article, Stojmenović remarks that extensions to other measures and proportions are possible, but no mention is made of dimensions $n > 3$.

The purpose of this article is to give a sufficient criterion for the existence of a hyperplane splitting continuous probability measures $\mu_1,\ldots,\mu_n$ in $\RR^n$ according to prescribed ratios $\alpha_1,\ldots,\alpha_n\in[0,1]$. The idea is this: Given a ratio $\alpha_i\in[0,1]$ and a measure $\mu_i$ we define an auxiliary function $f_i$ that selects for any normal vector $v\in S^{n-1}$ a point $f(v)\in H$ from a hyperplane $H$ with $\mu_i(H^+)=\alpha_i$. We do not assume the continuity of the $f_i$, only that their image is bounded. Our main result is: if the sets $\Im f_1,\ldots,\Im f_n$ can be separated by hyperplanes, then there exists a hyperplane $H$ with $\mu_i(H^+)=\alpha_i$ for all $i$. One corollary is that if the supports of the measures $\mu_i$ can be separated, we can find a splitting hyperplane for \emph{any} ratios $\alpha_i$. Interesting about these results is that we are able to show the existence of the desired hyperplane despite the absence of symmetry that is usually afforded by the ratio $\frac{1}{2}$. The main tool we use is the theorem of Poincaré-Miranda.

The organization of this paper is as follows. In section \ref{sec:definitions} we give basic definitions. In section \ref{sec:f_i} we introduce the auxiliary functions $f_i$. Examples showing the non-existence of uneven splittings in general are given in section \ref{sec:examples} which motivate the concept of separability defined in section \ref{sec:separation}. We then have everything in place to state our main result (theorem \ref{thm:main-result})in section \ref{sec:main-result}. Before turning to the proof in section \ref{sec:proof} we take a look at the main tool, the Poincaré-Miranda theorem, in section \ref{sec:topology}. There are many ways of choosing the auxiliary  functions $f_i$ and each gives rise to a different corollary of theorem \ref{thm:main-result}. For the purposes of theorem \ref{thm:main-result} only the sets $\Im f_i$ are of interest. In section \ref{sec:applications} we give a method of guaranteeing the existence of functions $f_i$ while controlling their image. We apply this method in section \ref{sec:two-lines} to generalize an easy corollary of the classical ham sandwich theorem on the partitioning of one mass by two lines. We conclude the paper in section \ref{sec:central-spheres} by suggesting a canonical choice for the functions $f_i$. 

\section{Definitions}
\label{sec:definitions}

Given an oriented hyperplane $H$, we denote the closed half-space in direction of the normal vector by $H^{+1}$ or $H^+$ and the closed half-space in direction of the reverse normal vector by $H^{-1}$ or $H^-$. The corresponding open half-spaces are denoted by $\interior{H}^+$ and $\interior{H}^-$, respectively. The hyperplanes and half-spaces in $\RR^n$ can be parametrized by pairs $(v,\lambda)$ of a normal vector $v\in\RR^n\setminus\{0\}$ and a scalar $\lambda\in\RR$ in the following fashion:
\begin{eqnarray*}
H_{v,\lambda}&=&\{x\in\RR^n : \sprod{x}{v}=\lambda \}\\
H^+_{v,\lambda}&=&\{x\in\RR^n : \sprod{x}{v}\geq\lambda \}\\
H^-_{v,\lambda}&=&\{x\in\RR^n : \sprod{x}{v}\leq\lambda \}
\end{eqnarray*}

By a \dfn{continuous probability measure} $\mu$ on $\RR^n$ we mean a measure $\mu$ on the set $\RR^n$ equipped with the Borel $\sigma$-algebra such that $\mu(\RR^n)=1$ and $\mu$ is absolutely continuous with respect to the Lebesgue measure $L$, i.e.\ if $L(S)=0$, then $\mu(S)=0$ for all measurable sets $S$. Whenever we use the term ``measure'', we assume that all of these conditions hold. Note that for a continuous measure the map $(v,\lambda)\mapsto \mu(H^+_{v,\lambda})$ is continuous. Note also that any continuous measure $\mu$ can be written as $\mu(S)=\int_S h dL$ where $h$ is a Borel-measurable function by the Radon-Nikodym theorem. Two $h_1,h_2$ with this property have to agree almost everywhere. We think of $\mu$ as coming with a fixed choice of $h$ and in abuse of terminology define the \dfn{support} of a measure as the set where $h$ is non-zero.

The question we want to address is the following.

\begin{question}
Let $\mu_1,\ldots,\mu_n$ be continuous probability measures on $\RR^n$ and $\alpha_1,\ldots,\linebreak\alpha_n\in[0,1]$. When does there exist a hyperplane $H_{v,\lambda}$ such that $\mu_i(H^+_{v,\lambda})=\alpha_i$ for all $i\in[n]$?
\end{question}

We call such a hyperplane $H_{v,\lambda}$ an \dfn{$(\alpha_1,\ldots,\alpha_n)$-splitting} of $\mu_1,\ldots,\mu_n$. If $\alpha_i\not=\frac{1}{2}$ for some $i$, we call the splitting \dfn{uneven}. It is well known that for a given uneven $\alpha$,  $\alpha$-splittings do not exist in general. Two examples for this are presented in section \ref{sec:examples}. 

\section{Auxiliary functions $f_i$}
\label{sec:f_i}

Before we turn to these examples we introduce a tool to keep track of the set of hyperplanes $H$ with $\mu(H)=\alpha$, for a fixed measure $\mu$ and a fixed ratio $\alpha\in[0,1]$. Unless the support of $\mu$ is unbounded and $\alpha\in\{0,1\}$, we always have that for any $v\in\RR^n\setminus\{0\}$ there exists a hyperplane $H_{v,\lambda}$ with $\mu(H^+)=\alpha$. The idea is to pick a point $f(v)\in H_{v,\lambda}$ to record its position. More formally, we let $f:\RR^n\setminus\{0\}\rar\RR^n$ be a function with $\mu(H^+_{v,\sprod{v}{f(v)}})=\alpha$ for all $v$. Note that $H_{v,\sprod{v}{x}}$ is the unique hyperplane with normal $v$ that contains $x$. Note also that in general, we do not require $f$ to be continuous.

\begin{example}
\label{ex:setup}
As an example consider the probability measure $\mu$ defined by the probability density that is constant non-zero on a ball in $\RR^n$ with center $c$ and radius $r$ and zero everywhere else. Let $0<\alpha<\frac{1}{2}$. In this case we can find an $f$ such that $f(v)$ always lies on a line through $c$ and has the desired property. Moreover $f:S^n\rar\RR^n$ is continuous and differentiable, $\Im f$ is a sphere with center $c$ and radius strictly smaller than $r$ and $H_{v,\sprod{v}{f(v)}}$ is tangent to $\Im f$ with $H^-_{v,\sprod{v}{f(v)}}\supset\Im f$. These additional properties of $f$ aid intuition, however, we make use of them only in the examples.
\end{example}

To illustrate the use of such a map $f$, we give the following proposition.

\begin{prop}
Let $\mu_1,\ldots,\mu_n$ be continuous measures, $\alpha_1,\ldots,\alpha_n\in(0,1)$ such that for all $i$ and all $v\in S^n$ the choice of $\lambda$ such that $\mu_i(H^+_{v,\lambda})=\alpha_i$ is unique. Let $f_i:S^n\rar\RR^n$ be such that $\mu_i(H^+_{v,\sprod{v}{f(v)}})=\alpha_i$ for all $i$.

Then there is a hyperplane $H$ such that $\mu_i(H^+)=\alpha_i$ for all $i$ if and only if there exists a $v\in S^n$ and a hyperplane $H$ with normal $v$ such that $f_i(v)\in H$ for all $i$. 
\end{prop}

\begin{proof}
``$\rAr$'' Let $v$ be the normal of $H$. By assumption there is only one hyperplane with normal $v$ that gives the desired split. So all $f_i(v)$ have to be contained in $H$.

``$\lAr$'' If there are $v$ and $H$ with normal $v$ and $f_i(v)\in H$ for all $i$, then $H=H_{v,\sprod{v}{f_i(v)}}$ for all $i$ and hence $\mu_i(H^+)=\alpha_i$ for all $i$ as desired.
\end{proof}

\section{Sandwiches Without Uneven Splittings}
\label{sec:examples}

We are now going to take an informal look at two examples where there are no $\alpha$-splittings, for certain uneven $\alpha$. These examples are built from spheres as in example \ref{ex:setup}. Note that in this case, for $\alpha\in(0,1)$ the normal determines the hyperplane whose positive half-space has measure $\alpha$ uniquely.

\begin{example}
Let $\mu_1$ and $\mu_2$ be two measures in $\RR^2$ given by two concentric discs of different radius. Say the radius of the disc of $\mu_1$ is larger. Fix $\alpha_1=\alpha_2=\alpha\in(0,\frac{1}{2})$ and consider the corresponding $f_1,f_2$. Both $\Im f_1$ and $\Im f_2$ are circles, the former containing the latter. Any hyperplane touching $f_1$ cannot touch $f_2$. More precisely: for any $v$ the point $f_2(v)$ lies in  $H^-_{v,\sprod{v}{f_1(v)}}$, at non-zero distance from $H_{v,\sprod{v}{f_1(v)}}$. Hence, there is no $\alpha$-splitting.
\end{example}

\begin{example}
Let $\mu_1,\mu_2,\mu_3$ be given by 3 disjoint balls in $\RR^3$ such  that their center-points lie on a line and the radius of the center sphere (say $\mu_2$) is larger than that of the other two. Fix one $\alpha\in(0,\frac{1}{2})$ for all $\mu_i$. The resulting $f_i$ are again spheres and the center one has larger radius than the other two. Now consider any hyperplane tangent to $\Im f_1$ and $\Im f_2$ (more precisely a hyperplane $H_{v,\lambda}$ such that $f_1(v),f_2(v)\in H_{v,\lambda}$). This hyperplane cannot meet $\Im f_3$, so there is no hyperplane splitting all three measures according to the desired ratio.
\end{example}

\section{Separability}
\label{sec:separation}

We have seen that if the $\Im f_i$ lie on a line, we cannot expect there to be an uneven splitting. However if the $\Im f_i$ were to be in ``general position'', we might have more luck. To make this notion precise we define the concept of sets in $n$-space that can be separated by hyperplanes.

$n$ points in $\RR^n$ are affinely independent, if and only if for any partition of these points into two classes, there is an affine hyperplane such that the one class is on one side and the other class is on the other side. In analogy, we say that sets $S_1,\ldots,S_n\subset\RR^n$ can be strictly separated by hyperplanes, or \dfn{separated} for short, if for any function $\sigma:[n]\rar\{-1,+1\}$, there  is a hyperplane $H$ such that $\interior{H^{\sigma(i)}}\supset S_i$ for all $i$. In other words, no matter how we prescribe which $S_i$ are supposed to be in front of $H$ and which $S_i$ are supposed to be behind $H$, we can always find a hyperplane $H$ that achieves this.

To illustrate this definition give the following proposition.
\begin{prop}
\begin{enumerate}
\item Sets $S_1,\ldots,S_n\in\RR^n$ can be separated if and only if $\conv S_1,\ldots, \linebreak\conv S_n$ can be separated.
\item Convex sets $S_1,\ldots, S_n\in\RR^n$ can be separated if and only if there does not exist an $(n-2)$-dimensional affine subspace that meets all of the $S_i$.
\item Points $s_1,\ldots,s_n\in\RR^n$ are in general position if and only if the singletons $\{s_1\},\ldots,\{s_n\}\subset\RR^n$ can be separated.
\end{enumerate}
\end{prop}

1.\ is immediate from the definition. 2.\ can e.g.\ be found in \cite[p.\ 218]{m-ldg}. 3.\ is an immediate consequence of 2.

\section{The Main Result}
\label{sec:main-result}

\begin{theorem}
\label{thm:main-result}
Let $\mu_1,\ldots,\mu_n$ be continuous probability measures on $\RR^n$. Let $\alpha_1,\ldots,\linebreak\alpha_n\in[0,1]$.  Let $f_1,\ldots,f_n:S^{n-1}\rar\RR^n$ be functions such that $\mu_i(H^+_{v,\sprod{v}{f_i(v)}})=\alpha_i$ for all $i\in[n]$ and all $v\in S^{n-1}$. The $f_i$ need not be continuous, but the $\Im f_i$ have to be bounded.

If $\Im f_1,\ldots,\Im f_n$ can be separated by hyperplanes, then there exists a hyperplane $H$ such that $\mu_i(H^+)=\alpha_i$ for all $i\in[n]$.
\end{theorem}

The strength of this theorem obviously stands and falls with our ability to find suitable functions $f_i$. We will have a closer look at this issue in sections \ref{sec:applications} and \ref{sec:central-spheres} where we present two ways of choosing the $f_i$, that give rise to corollaries \ref{cor:S-set} and \ref{cor:support}, respectively. See also section \ref{sec:central-spheres}.

\pagebreak
\section{The Topological Tool}
\label{sec:topology}

The topological tool we are going to use in the proof of theorem \ref{thm:main-result} is the Poincaré-Miranda theorem. Let $\SS^n$ denote the set of all partial functions $\sigma:[n]\rar \{-1,1\}$. Let $Q^n=[-1,1]^n$ denote the $n$-dimensional cube. For any $\sigma\in\SS^n$ let $Q^n_\sigma=Q^n\cap\{(x_1,\ldots,x_n):  x_i=\sigma(i)\  \forall i\in\dom\sigma\}$ denote a (closed) face of $Q^n$. If $|\dom\sigma|=n$, $Q^n_\sigma$ is a vertex. If $|\dom\sigma|=1$, $Q^n_\sigma$ is a facet. If $\dom\sigma=\emptyset$, $Q^n_\sigma=Q^n$.

\begin{block}{Poincaré-Miranda Theorem}
\label{thm:p-m}
Let $q:Q^n\rar \RR^n$ be a continuous function such that for all $\sigma\in\SS^n$ and all $i\in\dom\sigma$
$$ \sigma(i) q_i(Q^n_\sigma)\geq 0. $$

Then there exists a $x\in Q^n$ such that $q(x)=0$.
\end{block}

Note that $\sigma(i)q_i(Q^n_\sigma)\geq 0$ if and only if the following holds:
if $\sigma(i)=+1$ then $q_i(x)\geq 0$ for all $x\in Q^n_\sigma$ and if $\sigma(i)=-1$ then $q_i(x)\leq 0$ for all $x\in Q^n_\sigma$. 
Note also that we can, equivalently, restrict ourselves to those $\sigma$ with $|\dom\sigma|=1$ in the statement of the theorem. 

The Poincaré-Miranda theorem was shown in 1886 by Poincaré \cite{p-1886}. In 1940 Miranda \cite{m-40} showed it to be equivalent to Brouwer's fixpoint theorem from 1911. It is a beautiful generalization of the intermediate value theorem of Bolzano to higher dimensions, and it appears to be somewhat under-appreciated in modern topological geometry. For more information on this theorem we refer the reader to \cite[Chapter 4]{i-81} and \cite{k-97}.

The method by which we are going to apply this tool to the proof of our theorem is encapsulated in the following lemma which may be viewed as a slight generalization of the Poincaré-Miranda theorem. It is stated and proved here in a more general version than is necessary for the proof of theorem \ref{thm:main-result}, in the hope that this method may be useful in other contexts as well.

\begin{lemma}
\label{lem:key-lemma}
Let $X$ be a topological space, $p:X\rar \RR^n$ be a continuous function and $(C_\sigma^*)_{\sigma\in\SS^n}$ a family of sets with the following properties:
\begin{enumerate}
\item $C_\sigma^*\subset X$ for all $\sigma\in\SS^n$.
\item $C_\sigma^*$ is $(n-|\dom \sigma|-1)$-connected for all $\sigma\in\SS^n$.
\item If $\sigma_1\subset\sigma_2$, then $C_{\sigma_1}^*\supset C_{\sigma_1}^*$ for all $\sigma_1,\sigma_2\in\SS^n$. 
\item $\sigma(i) p_i(C^*_\sigma)\geq 0$ for all $\sigma\in\SS^n$ and all $i\in[n]$.
\end{enumerate}
Then there is an element $y\in C^*_\emptyset$ such that $p(y)=0$.
\end{lemma}

\begin{proof}
\emph{Strategy.} We are going to show that there is a continuous map $q:\nolinebreak Q^n\rar X$ such that $\sigma(i)(p_i\circ q)(Q^n_\sigma)\geq 0$ for all $\sigma$ with support of size 1 and for all $i\in[n]$. Then we are done by the Poincaré-Miranda theorem, as these conditions allow us to conclude that there exists an $x\in Q^n$ such that $(p\circ q)(x)=0$, which means that $y:=q(x)\in C_\emptyset^*$ is an element with $p(y)=0$.

To show this claim we proceed by induction over $|\dom \sigma|$. For each $\sigma\in\SS^n$ we are going to define $q|_{Q^n_\sigma}$ such that $q(Q^n_\sigma)\subset C^*_\sigma$. The foundation of the induction is the definition of $q(Q^n_\sigma)$ for $\sigma$ with $\dom(\sigma)=[n]$, i.e. placing the corners of the cube.

\emph{Placing the corners.} For each $\sigma$ with $\dom \sigma=[n]$ the set $Q^n_\sigma$ is a singleton $\{x_\sigma\}$. The set $C^*_\sigma$ is $-1$-connected which means non-empty, so we can pick a point $y_\sigma\in C^*_\sigma$ arbitrarily and define $q(x_\sigma):=y_\sigma$. We now have a defined a (continuous) function $q$ on the 0-skeleton of $Q^n$.

\emph{Induction over the size of $\dom \sigma$.} Let $\sigma$ be any sign function with $|\dom\sigma|=k$. Let $\sigma_1,\ldots,\sigma_l$ denote those sign functions with $\sigma_i\supset \sigma$ and $\dom\sigma_i=k+1$. By induction $q|_{\bigcup_{i=1}^lQ^n_{\sigma_i}}$ is defined and continuous. We are now going to extend this definition to $Q^n_\sigma$. 

Note that $Q^n_\sigma$ is an $n-k$-dimensional ball and $\bigcup_{i=1}^lQ^n_{\sigma_i}=\partial Q^n_\sigma$ is an $n-k-1$-dimensional sphere. Because $\sigma_i\supset\sigma$ we know by assumption that $C^*_{\sigma_i}\subset C^*_\sigma$. So $q(\bigcup_{i=1}^lQ^n_{\sigma_i})\subset C^*_\sigma$. Because $C^*_\sigma$ is $(n-k-1)$-connected, we can extend the definition of $q$ continuously to $Q^n_\sigma$. We proceed in the same way with all $\sigma$ with $|\dom \sigma|=k$ to obtain a continuous function $q$ defined on the $n-k$-skeleton of $Q^n$.

We continue until we have defined $q$ on $\partial Q^n$ using all the $\sigma$ with $|\dom \sigma|=1$. The image of $q|_{\partial Q^n}$ lies in $C^*_\emptyset$ which is $n-1$-connected, so we finish the construction by extending $q$ continuously to all of $Q^n$. 

\emph{Conclusion.} $q$ is continuous and defined on all of $Q^n$. Its image lies in $C^*_\emptyset\subset \dom p$. Now let $\sigma$ be any sign function with $|\dom \sigma|=1$.  By assumption $\sigma(i)p_i(C^*_\sigma)\geq0$ and as $q(Q^n_\sigma)\subset C^*_\sigma$ we conclude that $\sigma(i)(p_i\circ q)(Q^n_\sigma)\geq 0$ for all $i\in\dom\sigma$. So $p \circ q$ meets the conditions of the Poincaré-Miranda theorem.
\end{proof}

\section{Proof of Main Result}
\label{sec:proof}
\vspace{-2ex}

The strategy is to apply lemma \ref{lem:key-lemma}. To that end we will first define $p$ and then go on to define sets $C_\sigma$ as an intermediate step towards the definition of the sets $C_\sigma^*$.

\paragraph{Definition of $p$.}

A pair $(v,\lambda)\in \RR^n\times\RR$ defines an oriented hyperplane in $\RR^n$ if $v\not=0$. If $v=0$ we have no interpretation for $(v,\lambda)$. Hence we put $X=\RR^n\setminus\{0\}\times\RR$. Note that we can assume the $f_i$ to be defined on all of $X$ and not only on $S^{n-1}$ as we can extend them to the larger domain via $f_i(v)=f_i(v/ |v|)$, keeping $\Im f_i$ unchanged.
Now we define
\begin{eqnarray*}
  p:X & \rar & \RR^n\\
(v,\lambda) & \mapsto & (\mu_i(H^+_{v,\lambda})-\alpha_i)_{i\in[n]}
\end{eqnarray*}
which is a continuous function with the property that $p(v,\lambda)=0$, if and only if $\mu_i(H^+_{v,\lambda})=\alpha_i$ for all $i\in[n]$.

\paragraph{Definition and properties of $C_\sigma$.}

For each $\sigma\in\SS^n$ with $|\dom\sigma|=n$ let $H_\sigma$ denote a hyperplane with $\interior H_\sigma^{\sigma(i)}\supset S_i$ for all $i\in[n]$. Such a hyperplane exists by assumption. Let $\mathcal{H}'$ denote the arrangement of all such $H_\sigma$. The sets $\Im f_i$ are all contained in distinct $n$-dimensional cells of $\mathcal{H}'$. By assumption the $\Im f_i$ are bounded, so we can add hyperplanes to $\mathcal{H}'$ to obtain a hyperplane arrangement $\mathcal{H}$ in which each $\Im f_i$ is contained in a bounded $n$-dimensional cell. Let $S_i$ denote the closed $n$-dimensional cell in $\mathcal{H}$ containing $\Im f_i$. The sets $S_1,\ldots,S_n\subset \RR^n$ are $n$-dimensional polytopes. Let $V(S_i)$ denote their vertex sets.

For all partial functions $\sigma:[n]\rar \{-1,+1\}$ we define
$$C_\sigma :=\{(v,\lambda):\forall i\in \dom\sigma\ \forall x\in V(S_i): \sigma(i)\sprod{v}{x} \geq \sigma(i)\lambda\}.$$

We now check whether properties 1.\ through 4.\ from lemma \ref{lem:key-lemma} hold for the family $(C_\sigma)_{\sigma\in\SS^n}$. First of all $(C_\sigma)_{\sigma\in\SS^n}$ satisfies condition 4.\ as lemma \ref{lem:C4} shows.

\begin{lemma}
\label{lem:C4}
If $(v,\lambda)\in C_\sigma$, then $\sigma(i)p_i(v,\lambda)\geq 0$ for all $i\in\dom\sigma$.
\end{lemma}

\begin{proof}
Let $(v,\lambda)\in C_\sigma$ and $i\in\dom\sigma$. For any $x\in\RR^n$ the inequality $\sigma(i)\sprod{v}{x} \geq \sigma(i)\lambda$ holds, if any only if $x\in H^{\sigma(i)}_{v,\lambda}$. As the $S_i$ are convex $H^{\sigma(i)}_{v,\lambda}\supset S_i \supset \Im f_i$ follows and hence $f_i(v)\in H^{\sigma(i)}_{v,\lambda}$ which means $H^{\sigma(i)}_{v,\lambda}\supset H_{v,\sprod{v}{f_i(v)}}$. If $\sigma(i)=+1$ this implies $H^+_{v,\lambda}\supset H^+_{v,\sprod{v}{f_i(v)}}$, so $\mu_i(H^+_{v,\lambda})\geq\alpha_i$. If $\sigma(i)=-1$ this implies $H^+_{v,\lambda} \subset H^+_{v,\sprod{v}{f_i(v)}}$, so $\mu_i(H^+_{v,\lambda})\leq\alpha_i$. 
\end{proof}

That $(C_\sigma)_{\sigma\in\SS^n}$ satisfies 3.\ is immediate as $\sigma_1\subset\sigma_2$ means that any $(v,\lambda)\in C_{\sigma_1}$ has to satisfy a subset of the constraints defining $C_{\sigma_2}$. 2.\ also holds for the sets $C_\sigma$ as it turns out that they are non-empty and convex and hence simply connected. In fact we can show more:

\begin{lemma}
$C_\sigma$ is a polyhedral cone, for all $\sigma\in\SS^n$. If $\dom\sigma\not=\emptyset$, then $C_\sigma$ is pointed. For all $\sigma\in\SS^n$ there exists $(v,\lambda)\in C_\sigma$ with $v\not=0$.
\end{lemma}

\begin{proof}
\emph{$C_\sigma$ is a polyhedron.}
$C_\sigma$ is defined by a finite system of linear inequalities.

\emph{$C_\sigma$ is a cone.} For any $\beta\geq 0$ the implication
$$\sigma(i)\sprod{v}{x} \geq \sigma(i)\lambda \rAr \sigma(i)\sprod{\beta v}{x} \geq \sigma(i)\beta \lambda $$
holds and hence if $(v,\lambda)\in C_\sigma$, then $\beta(v,\lambda)\in C_\sigma$.

\emph{$C_\sigma$ is pointed.} Suppose $i\in\dom\sigma$ and both $(v,\lambda)$ and $-(v,\lambda)$ are in $C_\sigma$, then for all $x\in S_i$ we have
\begin{eqnarray*}
 \sigma(i)\sprod{v}{x} \leq \sigma(i)\lambda&\wedge&  \sigma(i)\sprod{v}{x} \geq \sigma(i)\lambda
\end{eqnarray*}
so in fact equality holds. This implies that $S_i$ lies in the hyperplane defined by this equation, which is a contradiction to $S_i$ being $n$-dimensional.

\emph{There exists $(v,\lambda)\in C_\sigma$ with $v\not=0$.} Let $\sigma\in\SS^n$. Let $S_i'$ denote the $n$-dimensional cells of $\mathcal{H}'$ containing $\Im f_i$ for all $i\in[n]$. By definition of $\mathcal{H}'$ there is a hyperplane $H$ such that $H^{\sigma(i)}\supset S_i'\supset S_i$ for all $i\in\dom\sigma$. 

\end{proof}

Note that $(0,0)\in C_\sigma$ for all $\sigma\in\SS^n$, $\{(0,\lambda):\lambda\geq 0\}\subset C_\sigma$ iff $\Im \sigma\subset\{-1\}$ and $\{(0,\lambda):\lambda\leq 0\}\subset C_\sigma$ iff $\Im \sigma\subset\{+1\}$. So the sets $C_\sigma$ meet conditions 2., 3.\ and 4., but they do not meet condition 1. However, if we were to correct this by removing $\{(0,\lambda):\lambda\in\RR\}$ we would destroy the connectivity of the $C_\sigma$ and violate condition 2. Note that $X$ itself is only $n-2$-connected and not $n-1$-connected. So we have to use a different construction.

\paragraph{Definition and properties of the $C^*_\sigma$.}

For each $i\in[n]$, let $b_i$ be the barycenter of $S_i$. Let $H_{v_0,\lambda_0}$ be an oriented hyperplane containing all the $b_i$. This exists, because $b_1,\ldots,b_n$ are only $n$ points $\RR^n$. Define
$$C^*_\sigma:= C_\sigma\setminus \{(\beta v_0,\lambda): \lambda\in \RR,\beta\geq 0\}.$$
Note that $\{(\beta v_0,\lambda): \lambda\in \RR,\beta\geq 0\}\supset\{(0,\lambda):\lambda\in\RR\}$ and hence $C^*_\sigma\subset X$, so these sets fulfill condition 1. As $C^*_\sigma\subset C_\sigma$, condition 4.\ still holds. As we removed the same set from all the $C_\sigma$, we also have condition 3. What is not immediately clear, is whether 2.\ holds. We are going to show that the $C^*_\sigma$ are simply connected, and hence $(n-|\dom \sigma|-1)$-connected as required.

\begin{lemma}
For any $\sigma\in\SS^n$ the set $C^*_\sigma$ is simply connected.
\end{lemma}

\begin{proof}
\emph{Case 1: $|\Im\sigma|=2$.} 
We are going to show that $C^*_\sigma$ is non-empty and convex, which implies that $C^*_\sigma$ is simply connected.

We start out by showing that $C^*_\sigma=C_\sigma\cap X$. Suppose that $(\beta v_0,\lambda)\in C_\sigma$ for some $\beta>0$. Let, say, $\sigma(1)=+1$ and $\sigma(2)=-1$. Then $b_1\in S_1\subset H^+_{\beta v_0,\lambda}$ and $b_2\in S_2\subset H^-_{\beta v_0,\lambda}$, so $\lambda\leq\sprod{\beta v_0}{b_1}=\sprod{\beta v_0}{b_2}\leq \lambda$ where the equality holds by the choice of $v_0$.  So $b_1,b_2\in H_{\beta v_0,\lambda}$. But the $S_i$ are $n$-dimensional. So there exist $x_1,x_2\in S_1$ with $\sprod{\beta v_0}{x_1} < \lambda < \sprod{\beta v_0}{x_2}$, which is a contradiction to $S_1\subset H^+_{\beta v_0,\lambda}$.

We now know that $C^*_\sigma=C_\sigma\cap X$ and as we have already seen, $C_\sigma$ contains a pair $(v,\lambda)$ such that $v\not=0$, so we can conclude that $(v,\lambda)\in C^*\sigma$. Also we know that $C_\sigma$ is convex and we want to show that $C^*_\sigma$ is convex as well. It suffices to argue that there are no $v,\lambda,\lambda'$ such that both $(v,\lambda)$ and $(-v,\lambda')$ are in $C_\sigma$. Assume to the contrary that this is the case. Then 
$$H^+_{v,\lambda}\supset S_1 \subset H^+_{-v,\lambda'} \text{ and }
H^-_{v,\lambda}\supset S_2 \subset H^-_{-v,\lambda'}.
$$
But 
$$H^+_{-v,\lambda'}=H^-_{v,-\lambda'} \text{ and } H^-_{-v,\lambda'}=H^+_{v,-\lambda'},$$
so
$$S_1\subset H^+_{v,\lambda} \cap H^-_{v,-\lambda'} \text{ and }
S_2 \subset H^-_{v,\lambda} \cap H^+_{v,-\lambda'},$$
and hence
$$\exists s_1\in S_1: \lambda \leq \sprod{v}{s_1} \leq -\lambda' \text{ and }
\exists s_2\in S_2: -\lambda' \leq \sprod{v}{s_1} \leq \lambda.$$
We conclude that $\lambda= -\lambda'$ which implies
$$S_1\subset H^+_{v,\lambda} \cap H^-_{v,\lambda}=H_{v,\lambda} \text{ and }
S_2 \subset H^-_{v,\lambda} \cap H^+_{v,\lambda}=H_{v,\lambda}.$$
But the $S_i$ are full dimensional and cannot be contained in a hyperplane, which is a contradiction.

\emph{Case 2: $|\Im\sigma|\leq 1$.} 
Without loss of generality, we assume that $\Im\sigma\subset\{+1\}$. $\sigma$ may be empty. For $-v_0$ there exists a $\lambda'$ such that $H^+_{-v_0,\lambda'}\supset S_i$ for all $i$, because the $S_i$ are bounded. This shows that $C^*_\sigma$ is non-empty. We are going to give a deformation retraction of $C^*_\sigma$ to $(-v_0,\lambda')$. This means that we are looking for a homotopy $h_t:C^*_\sigma\rar C^*_\sigma$ such that $h_0=\text{id}$ and $h_1$ is the map with $h_1(v,\lambda) = (-v_0,\lambda')$ for all $(v,\lambda)$. We are going to use the straight line homotopy, defined by
$$h_t(v,\lambda)=t(-v_0,\lambda') + (1-t)(v,\lambda).$$
$h_t$ is continuous, $h_0(v,\lambda)=(v,\lambda)$ and $h_1(v,\lambda)=(-v_0,\lambda')$ as desired. We still have to show that $\Im h_t\subset C^*_\sigma$ for all $t$. $C_\sigma$ is a cone and hence convex, so all we have to argue is that no $(\beta v_0,\gamma)$ with $\beta \geq 0$ and $\gamma\in\RR$ lies in $\Im h_t$ for any $t$. We have already seen that $\Im h_0=C^*_\sigma$ and $\Im h_1=\{(-v_0,\lambda')\}$, so we only have to consider $0<t<1$.   Suppose $h_t(v,\lambda)=(\beta v_0,\gamma)$ for such $\beta$ and $\gamma$ for some $t\in(0,1)$. Then
$$-tv_0+(1-t)v=\beta v_0$$
which we can solve for $v$ to obtain
$$v=\frac{\beta+t}{1-t}v_0.$$
We have thus written $v$ as a non-negative multiple of $v_0$. But then $(v,\lambda)\not\in C^*_\sigma$. So $C^*_\sigma$ is contractible and hence simply connected.
\end{proof}

We have now seen that the $C^*_\sigma$ also meet condition 2.\ and hence all the requirements of lemma \ref{lem:key-lemma}. Applying lemma \ref{lem:key-lemma} gives us the desired $(v,\lambda)\in\RR^{n+1}$ with $v\not=0$ such that $p(v,\lambda)=0$. This concludes the proof of theorem \ref{thm:main-result}.

\section{Existence of the $f_i$}
\label{sec:applications}

\begin{prop}
\label{prop:existence-key-lemma}
Let $\mu$ be a continuous probability measure and $\alpha\in[0,1]$.
If $S$ is a compact, connected set with $\mu(S)\geq \max\{\alpha,1-\alpha\}$, then there exists a function $f:S^{n-1}\rar\RR^n$ such that $\mu(H^+_{v,\sprod{v}{f(v)}})=\alpha$ for all $v\in S^{n-1}$ and $\Im f \subset S$.
\end{prop}

\begin{proof}
First of all, note that we can assume $\alpha\leq\frac{1}{2}$ without loss of generality. By definition a set $S$ with the given properties exists for a parameter $\alpha$ if and only if it exists for a parameter $1-\alpha$. Given an auxiliary function $f$ with $\mu(H^+_{v,\sprod{v}{f(v)}})=\alpha$ for all $v\in S^{n-1}$ and $\Im f \subset S$, we can obtain an auxiliary function $f'$ with $\mu(H^+_{v,\sprod{v}{f'(v)}})=1-\alpha$ for all $v\in S^{n-1}$ and $\Im f' \subset S$ by putting $f'(v)=f(-v)$. Therefore, let $\alpha\leq \frac{1}{2} \leq 1-\alpha$.

Let $v\in S^{n-1}$. If we can show that there exists a hyperplane $H_{v,\lambda}$ with $\mu(H^+_{v,\lambda})=\alpha$ and $H_{v,\lambda}\cap S\not=\emptyset$ we are done. $S$ is compact and hence bounded. So there exist hyperplanes $H_{v,\lambda_+}$ and $H_{v,\lambda_-}$ such that $\mu(H^+_{v,\lambda_+})\geq 1-\alpha \geq \alpha \geq \mu(H^+_{v,\lambda_-})$. By continuity of $\mu$ there exists a hyperplane $H_{v,\lambda_0}$ with $\mu(H^+_{v,\lambda_0})=\alpha$. Now we distinguish three cases.

\emph{Case 1: $H^+_{v,\lambda_0}\supset S$.} Then $\alpha=\mu(H^+_{v,\lambda_0})\geq\mu(S)=1-\alpha$ and hence $\alpha=1-\alpha$ as well as $\mu(H^+_{v,\lambda_0})=\mu(S)$. Put $\lambda=\sup\{\lambda'\geq\lambda_0:H^+_{v,\lambda'}\supset S\}$. As $S$ is compact $H_{v,\lambda}\cap S\not=\emptyset$. $$\alpha=\mu(H^+_{v,\lambda_0})\geq\mu(H^+_{v,\lambda'})\geq \mu(S) = 1-\alpha =\alpha$$ 
for every $\lambda'\geq\lambda_0$ with $H^+_{v,\lambda'}\supset S$, so by continuity of $\mu$ we also have $\mu(H^+_{v,\lambda})=\alpha$.

\emph{Case 2: $H^-_{v,\lambda_0}\supset S$.} Then $1-\alpha=\mu(H^-_{v,\lambda_0})\geq\mu(S)=1-\alpha$ and hence $\mu(H^-_{v,\lambda_0})=\mu(S)$. Put $\lambda=\inf\{\lambda'\leq\lambda_0:H^-_{v,\lambda'}\supset S\}$. As $S$ is compact $H_{v,\lambda}\cap S\not=\emptyset$. $$\mu(H^-_{v,\lambda_0})\geq\mu(H^-_{v,\lambda'})\geq \mu(S) = 1-\alpha $$ 
for every $\lambda'\geq\lambda_0$ with $H^-_{v,\lambda'}\supset S$, so by continuity of $\mu$ we have $\mu(H^+_{v,\lambda})=1-\mu(H^-_{v,\lambda})=\alpha$.

\emph{Case 3: Neither inclusion holds.} Then $H^+_{v,\lambda_0}\cap S \not= \emptyset \not= H^-_{v,\lambda_0}\cap S$ and because $S$ is connected $H^-_{v,\lambda_0}\cap S\not = \emptyset$.
\end{proof}

If $0\not=\alpha\not=1$ or the support of $\mu$ is bounded, then a set $S$ as in proposition \ref{prop:existence-key-lemma} always exists. There are examples where the support of $\mu$ is unbounded such that for $\alpha\in\{0,1\}$ no function $f$ with $\mu(H^+_{v,\sprod{v}{f(v)}})=\alpha$ exists.

\begin{corollary}
\label{cor:S-set}
Let $\mu_1,\ldots,\mu_n$ be continuous probability measures on $\RR^n$. Let $\alpha_1,\ldots,\linebreak \alpha_n\in[0,1]$. Let $S_1,\ldots,S_n\subset\RR^n$ have one of the following two properties:
\begin{enumerate}
\item $S_1,\ldots,S_n$ are compact, can be separated and $$\mu_i(S_i)\geq\max\{\alpha_i,1-\alpha_i\} \text{ for all $i\in[n]$.}$$
\item $S_1,\ldots, S_n$ are closed, can be separated and
$$\mu_i(S_i)>\max\{\alpha_i,1-\alpha_i\} \text{ for all $i\in[n]$.}$$
\end{enumerate}
Then there exists a hyperplane $H$ such that $\mu_i(H^+)=\alpha_i$ for all $i\in[n]$.
\end{corollary}

\begin{proof}
Suppose $S_1,\ldots, S_n$ are compact, can be separated and $\mu_i(S_i)\geq\max\{\alpha_i,\linebreak 1-\alpha_i\}$ for all $i\in[n]$. We first pass to the respective convex hulls $S_i'=\conv(S_i)$ which are compact, convex and hence connected, can be separated and have $\mu_i(S'_i)\geq\max\{\alpha_i,1-\alpha_i\}$. Applying \ref{prop:existence-key-lemma} and \ref{thm:main-result} shows that condition 1.\ suffices for the existence of the desired hyperplane.

If $S_1,\ldots, S_n$ are closed, can be separated and
$\mu_i(S_i)>\max\{\alpha_i,1-\alpha_i\}$ for all $i$, then for every $i$ there exists a ball $B_i$ such that $S_i':=S_i\cap B_i$ is compact and $\mu_i(S_i')>\max\{\alpha_i,1-\alpha_i\}$. $S_1',\ldots, S_n'$ can be separated. Using the sufficiency of 1.\ we obtain that 2.\ is sufficient for the existence of the desired hyperplane.
\end{proof}

Let $S$ denote the closure of the support of $\mu$. If $S$ is bounded, then $S$ fulfills the conditions of corollary \ref{cor:S-set} for any choice of $\alpha$. This immediately gives rise to the following result.

\begin{corollary}
\label{cor:support}
Let $\mu_1,\ldots,\mu_n$ be continuous probability measures  on $\RR^n$. Let $S_1,\ldots,\linebreak S_n$ denote the closure of their respective support.

If $S_1,\ldots, S_n$ are bounded and can be separated, then for any ratios $\alpha_1,\ldots,\alpha_n\in[0,1]$, there exists a hyperplane $H$ such that $\mu_i(H^+)=\alpha_i$ for all $i\in[n]$.
\end{corollary}

As mentioned in the introduction \cite{stoj1} contains a remark implying \ref{cor:support}. However we were not able to find a rigorous proof of \ref{cor:support} in full generality anywhere in the literature. Our method of proof differs from the one implied by the algorithm from \cite{stoj1}.

\section{Partitioning one Mass by two Lines}
\label{sec:two-lines}

It is an easy corollary of the ham sandwich theorem in the plane, that any continuous probability measure in the plane can be partitioned by two lines into four parts of measure $\frac{1}{4}$. In just the same way corollary \ref{cor:S-set} implies the following proposition.

\begin{prop}
For any continuous probability measure $\mu$ on $\RR^n$ and any $\alpha_1,\ldots,\linebreak\alpha_4\in(0,1)$ with $\sum_i \alpha_i=1$, there exist two oriented hyperplanes $H_1,H_2$ such that $H_1^+\cap H_2^+$, $H_1^-\cap H_2^+$, $H_1^+\cap H_2^-$, $H_1^-\cap H_2^-$ have $\mu$-measure $\alpha_1,\ldots,\alpha_4$, respectively.
\end{prop}

\begin{proof}
Without loss of generality, we can assume $n=2$. For higher $n$ we pick an arbitrary projection down to $\RR^2$. Now we pick any normal $v$ and let $H_{v,\lambda}=:H_2$ denote a hyperplane with $\mu(H_{v,\lambda}^+)=\alpha_1+\alpha_2$. Now put $\mu_1(A)=\frac{1}{\alpha_1+\alpha_2}\mu (A\cap H^+_{v,\lambda})$ and $\mu_2(A)=\frac{1}{\alpha_3+\alpha_4}\mu (A\cap H^-_{v,\lambda})$. $\mu_1$ and $\mu_2$ are again continuous probability measures. Because all $\alpha_i$ are non-zero, there exist $\lambda_1>\lambda>\lambda_2$ such that $$\mu_1(H^+_{v,\lambda_1})>\max\{\frac{\alpha_1}{\alpha_1+\alpha_2},1-\frac{\alpha_1}{\alpha_1+\alpha_2}\}$$
and
$$\mu_2(H^-_{v,\lambda_2})>\max\{\frac{\alpha_3}{\alpha_3+\alpha_4},1-\frac{\alpha_3}{\alpha_3+\alpha_4}\}.$$
$H^+_{v,\lambda_1}$ and $H^-_{v,\lambda_2}$ are closed and they are separated by $H_{v,\lambda}$. Applying corollary \ref{cor:S-set} yields a hyperplane $H_1$ such that $\mu_1(H_1^+)=\frac{\alpha_1}{\alpha_1+\alpha_2}$ and $\mu_2(H_1^+)=\frac{\alpha_3}{\alpha_3+\alpha_4}$. By construction these two hyperplanes $H_1,H_2$ now have the desired properties.
\end{proof}

\pagebreak
\section{Central Spheres}
\label{sec:central-spheres}

The method given in section \ref{sec:applications} allows us to control $\Im f_i$, but it gives a lot of freedom for choosing the particular function $f_i$. What is a good canonical choice? We suggest what we call the central sphere, which we define below. First we need to recall some definitions. Given a density function $h$ on $\RR^k$, the \dfn{center of mass} of $h$ is defined as 
$$\int_{\RR^k} \text{id}\cdot h\ dL$$
where $L$ is the Lebesgue measure on $\RR^k$. Given a set $S\subset \RR^k$, we consider the affine subspace $A$ of minimal dimension that contains $S$. Let $L'$ denote the Lebesgue measure on $A$. The \dfn{centroid} of $S$ is defined as 
$$\frac{1}{L'(S)}\int_{S} id\ dL'$$
which we interpret as a point in $\RR^k$.

Let $\mu$ denote a continuous probability measure with bounded support $S$ and $h$ a fixed density function of $\mu$ (see section \ref{sec:definitions}). Let $\alpha\in[0,1]$. Let $S$ be a convex set that is compact and has $\mu(S)\geq\max\{\alpha,1-\alpha\}$ or that is closed and has $\mu(S)>\max\{\alpha,1-\alpha\}$. We now define a function $c:S^{n-1}\rar\RR^n$ which we call the \dfn{central sphere} of $\mu$, $h$, $S$ and $\alpha$ as follows. For each $v\in S^{n-1}$ there exists a hyperplane $H_{v,\lambda}$ with $\mu(H_{v,\lambda}^+)=\alpha$ and $H_{v,\lambda} \cap S\not=\emptyset$ as described in the proof of proposition \ref{prop:existence-key-lemma}. Let $\chi_S$ denote the characteristic function of $S$. If $(h\cdot \chi_S)|_{H_{v,\lambda}}=0$ almost everywhere with respect to the Lebesgue measure on $H_{v,\lambda}$, we define $c(v)$ to be the centroid of $H_{v,\lambda} \cap S$. Otherwise we define $c(v)$ to be the center of mass of $(h\cdot \chi_S)|_{H_{v,\lambda}}$.  With this definition $\mu(H^+_{v,\sprod{v}{c(v)}})=\alpha$ for all $v\in S^{n-1}$ and $\Im c\subset S$.

The central spheres are particular functions that can be used in the proof of corollary \ref{cor:S-set}. They give rise to a slightly stronger version of that corollary, in which we only need to require that the sets $\Im c_i$ can be separated, not the sets $S_i$. (Here $c_i$ denotes the central sphere of $\mu_i$, $h_i$, $S_i$ and $\alpha_i$.) Note that this makes only sense if we really fix a density function $h_i$ for each $\mu_i$ a priori: for each $\mu_1,\ldots,\mu_n$ there is a choice of density functions $h_1,\ldots,h_n$ such that if $\Im c_1,\ldots,\Im c_n$ can be separated, then $S_1,\ldots,S_n$ can be separated.

\begin{example}
Consider a regular 5-gon $P$ in the plane. Let $\mu(\Omega)$ be the area of $\Omega\cap P$. Let $h=\chi_P$. Let $S=\RR^2$. We consider the central spheres for $\alpha=\frac{1}{2}$ and for some small $\alpha>0$. Qualitative pictures of the corresponding central spheres are shown in figure \ref{fig:examples}, see also \cite{rote-05}. There are a couple of things to note. 1) In the case $\alpha=\frac{1}{2}$ the central sphere has Whitney index 4 (not 2), while in the case of small $\alpha$ the curve has Whitney index 1. 2) Even in the case $\alpha=\frac{1}{2}$ the hyperplane $H_{v,\lambda(v)}$ does not contain the centroid of $P$. So even in this case, the $f_i$ cannot be chosen to be constant. See \cite{bcj-01}. 3) The image of the central sphere is not in general the boundary of a convex set.
\end{example}

Some remarks about the continuity of $c$: The central sphere of $\mu,\alpha$ is not in general continuous, not even if the density $h$ of $\mu$ is smooth.  Moreover, even if we only define $c$ for those $v$ where $h|_H=0$ does not hold almost a.e., $c$ cannot be extended continuously to all $v$. An example for this is given in figure \ref{fig:non-cont}: On each of the three disks, the density of $\mu$ is an identical smooth cap. Now, if $\alpha=\frac{1}{3}$ and the normal $v$ points directly upwards, then $H:=H_{v,\lambda(v)}$ is as shown in the figure. On $H$ the density $h$ is identical $0$. Tilting $v$ slightly to the left by \emph{any} small amount, we always obtain a hyperplane that intersects $H$ in $p$. Tilting $v$ slightly to the right by \emph{any} small amount, we always obtain a hyperplane that intersects $H$ in $q$. No matter how we choose $c(v)$, we can never obtain a continuous function.

\begin{figure}[h]
\begin{center}
\includegraphics[width=6cm]{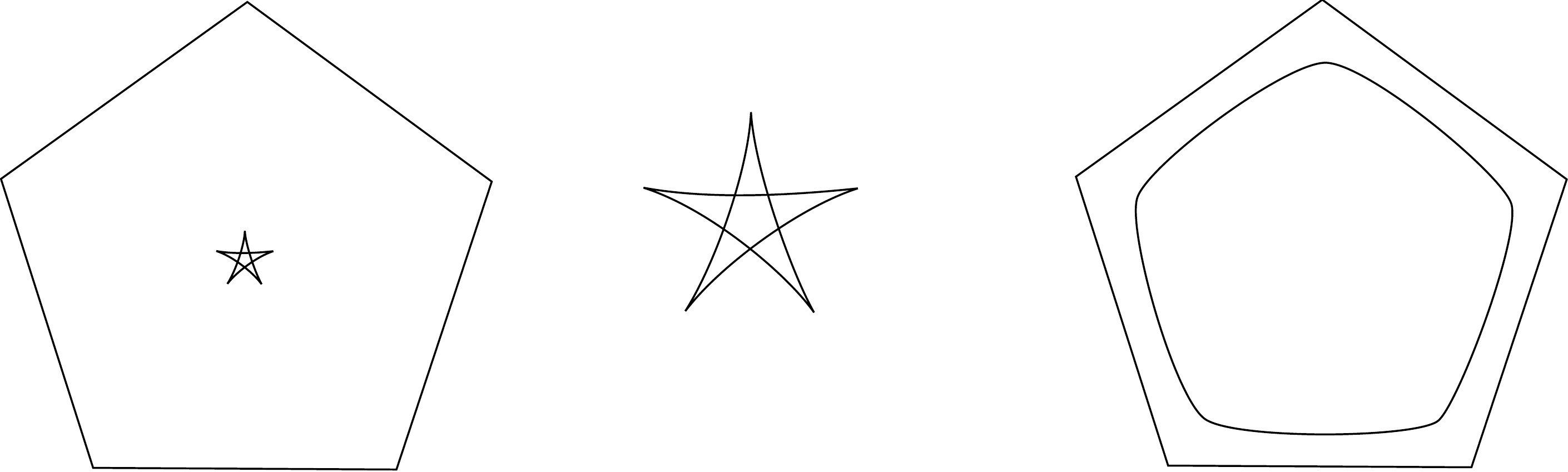}
\end{center}
\caption{\label{fig:examples}Central spheres of a 5-gon for $\alpha=\frac{1}{2}$ (with 5-gon and magnified) and small $\alpha>0$.}
\end{figure}

\begin{figure}[h]
\begin{center}
\includegraphics[width=8cm]{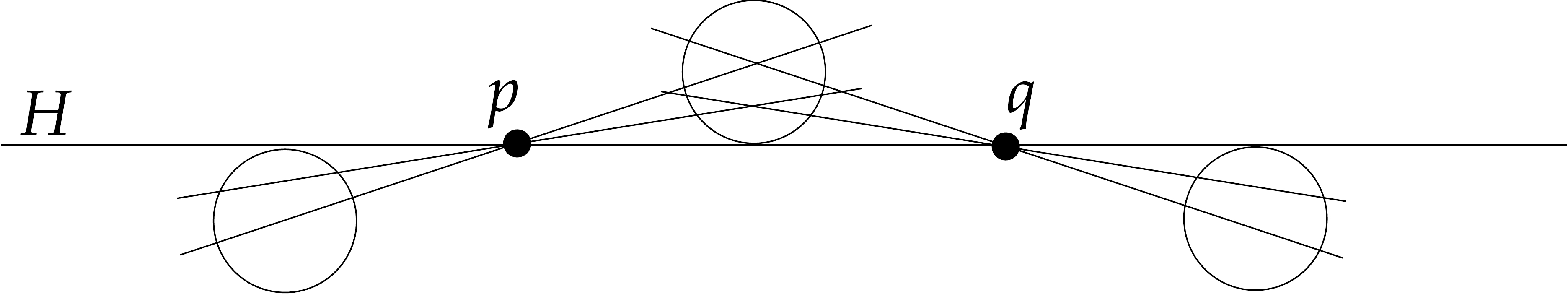}
\end{center}
\caption{\label{fig:non-cont}Central spheres are not continuous in general.}
\end{figure}

\end{document}